\pdfoutput=1
\RequirePackage{ifpdf}
\ifpdf 
\documentclass[pdftex]{sigma}
\else
\documentclass{sigma}
\fi

\numberwithin{equation}{section}

\newtheorem{Theorem}{Theorem}[section]
\newtheorem*{Theorem*}{Theorem}

\newtheorem{Lemma}[Theorem]{Lemma}
\newtheorem{Proposition}[Theorem]{Proposition}
 { \theoremstyle{definition}
\newtheorem{Definition}[Theorem]{Definition}

\newtheorem{Example}[Theorem]{Example}
 }

\usepackage{yfonts}

\usepackage{mathrsfs}
\usepackage{bm}

\usepackage{subcaption}
\usepackage[skip=5pt,labelfont=bf,labelsep=period]{caption}

\newcommand{\id}{\mathrm{id}}
\newcommand{\ev}{\mathrm{ev}}

\newcommand{\extd}{\mathrm{d}}
\newcommand{\tens}{\mathop{\otimes}}
\newcommand{\R}{\mathbb{R}}

\newcommand{\C}{\mathbb{C}}

\newcommand{\Z}{\mathbb{Z}}

\newcommand{\<}{\langle}
\renewcommand{\>}{\rangle}

\begin{document}
\allowdisplaybreaks

\renewcommand{\PaperNumber}{017}

\FirstPageHeading

\ShortArticleName{The Exponential Map for Hopf Algebras}

\ArticleName{The Exponential Map for Hopf Algebras}

\Author{Ghaliah ALHAMZI~$^{\rm a}$ and Edwin BEGGS~$^{\rm b}$}

\AuthorNameForHeading{G.~Alhamzi and E.~Beggs}

\Address{$^{\rm a)}$~Imam Mohammad Ibn Saud Islamic University (IMSIU), Riyadh, Saudi Arabia}
\EmailD{\href{mailto:gyalhamzi@imamu.edu.sa}{gyalhamzi@imamu.edu.sa}}

\Address{$^{\rm b)}$~Department of Mathematics, Swansea University, Wales, UK}
\EmailD{\href{mailto:e.j.beggs@swansea.ac.uk}{e.j.beggs@swansea.ac.uk}}
\URLaddressD{\url{https://www.swansea.ac.uk/staff/science/maths/beggs-e-j/}}

\ArticleDates{Received June 15, 2021, in final form February 16, 2022; Published online March 09, 2022}

\Abstract{We give an analogue of the classical exponential map on Lie groups for Hopf $*$-algebras with differential calculus. The major difference with the classical case is the inter\-pretation of the value of the exponential map, classically an element of the Lie group. We give interpretations as states on the Hopf algebra, elements of a Hilbert $C^{*} $-bimodule of $\frac{1}{2}$ densities and elements of the dual Hopf algebra. We give examples for complex valued functions on the groups $S_{3}$ and $\mathbb{Z}$, Woronowicz's matrix quantum group $\mathbb{C}_{q}[SU_2] $ and the Sweedler--Taft algebra.}


\Classification{16T05; 46L87; 58B32}

\section{Introduction}
For a Lie group $ G $ with Lie algebra $ \textgoth{g} $ we have the exponential map $ \exp\colon \textgoth{g} \rightarrow G $ \cite{Hausner}. We wish to give a Hopf algebra generalisation of this map. The first thing is to decide what maps to what, and the easy bit is what we map out of. Classically we have an algebra of smooth functions $ C^{\infty } (G, \C) $ with a $ * $-structure which is just pointwise conjugation. This has a differential calculus of 1-forms $ \Omega_{G}^{1} $ (again we take complex valued with a $ * $-operation). Then the left invariant $ 1 $-forms $ \Lambda_{G}^{1} $ (equivalently the cotangent space at the identity) has a dual $ \textgoth{g} _{\C}$, the complexified Lie algebra, and $ \textgoth{g} $ is the real part of $ \textgoth{g} _{\C} $. We apologise to real differential geometers for the seemingly unnecessary diversion through the complex numbers, but for Hopf algebras over $ \C $ this will become necessary. Algebraically, from the group $ G $ we take the algebra of complex valued smooth functions $ C^{\infty}(G) $, and then the differential calculus $ \Omega_{G}^{1} $ for this algebra.

The target for the exponential map is more of a problem. To explain, we ignore analytic complications (after all we will consider finite and discrete groups), and just take $ \C G $ to be the complex group algebra. We can view $g\in G $ in two ways, first as $g\in \C G $ and secondly as the state ``evaluate at $ g $'' on the function algebra $ C (G )$. To complicate things further, the state can be written using a GNS construction using a Hilbert space ($\frac{1}{2}$ densities) and we could consider~$ g $ to lie in this Hilbert space. All of these points of view will appear.

To motivate the exponential map for Hopf algebras we first look at the Lie group case. Here the motivation is obvious, we know an enormous amount about Lie algebras and their classification, and the exponential map allows us to use this to study Lie groups. Hopf algebras are at a different stage of their history; we know little about the classification theory. We know almost nothing about the non-bicovariant case, in fact the associative algebra $ U(\textgoth{g}) $ in this case was only worked out recently in terms of higher-order differential operators on Hopf algebras (see~\cite{Bgeholom} or \cite[p.~515]{BMbook}). (To clarify, here $U(\textgoth{g})$ denotes the algebra generated by invariant vector fields, and does not refer to a deformation of a classical enveloping algebra for which we use~$ U_{q}({\mathfrak{su}}_{2}) $ below.) Of course, for a Hopf algebra $ H $ the dual of $\textgoth{g}$ is a quotient of $ H^{+} $ by an ideal, but this just illustrates the problem in that we start with~$ H $ to find $\textgoth{g}$. Published material on quantum Lie algebras concentrates on their properties (e.g., \cite{GomMa:blie,Ma:blie}), especially in the braided case, and not on a classification theory. We would hope that giving a very general construction for the exponential map for Hopf algebras would motivate the study of the corresponding quantum Lie algebras in their own right, including their classification theory. In turn, this could be used in classification of Hopf algebras.

\looseness=-1We begin by using the Kasparov, Stinespring, Gel'fand, Na\u{\i}mark and Segal (KSGNS) construction (see \cite{Lance}) in a case which (for $ C^{*}$-algebras) gives a time varying state $ \psi_{t} $ on the function algebra $ C(G) $. The KSGNS construction works by using a bimodule as half densities, and we have~$ m(t) $ an element of this bimodule. The construction for~$ m(t) $ reduces to a differential equation, and solving this equation uses an actual exponential (Taylor series) in the group algebra~$ \C G $.

In Section~\ref{Ch1} we give a construction for the dynamics of states on an algebra from~\cite{Bgeo}. This uses a vector field on an algebra, and for a Hopf algebra we can use the left invariant vector fields in an exact correspondence with the classical case using Woronowicz's calculi on Hopf algebras~\cite{WoroMatl}. We will do all this for four examples, the discrete group $ S_{3} $ in the simplest case, and the group $ \Z $ is not much more complicated (though it gives an interesting non-diffusion evolution of states on~$ \Z $). The quantum group $ \C_{q}[SU_{2}] $ has much more complicated formulae, so we carry out calculations only in special cases. The Sweedler--Taft Hopf algebra is included to stress that the method is more general than its initial motivation in $ C^{*} $-algebras.

\looseness=-1 In solving the dynamics of the states we come on a much simpler and very direct interpretation of the exponential as a power series in the dual Hopf algebra. If we had simply written this in the beginning then questions would have been asked about its role and whether it was just writing a power series to look like the classical case. But now we can be more clear about its role, it plays a fundamental part in the dynamics of the states in the Hopf algebra using Hilbert $ C^{*} $-bimodules. Given an invariant vector field (i.e., an element of the ``Lie algebra'' of the Hopf algebra) we get an exponential path in time $ t\in \R $ lying in the dual Hopf algebra starting at~$ \epsilon $ (i.e., evaluation at the identity). Of course, for many Hopf algebras the exponential will not lie in the original dual Hopf algebra as it is an infinite series, but in a completion or formal extension.

Since we are exponentiating a vector field $ X $, the reader may be puzzled about why in various places (e.g., Proposition~\ref{P4}) we get an exponential of minus $ X $. The simple explanation is that the ``weight'' defining the functional moves in the opposite direction to what the functional applies to. Thus for a functional $ T_{t}\colon C_{0}(\R) \rightarrow \R $ and weight $ w(x) $
\[T_{t}(f)=\int_{\R} f(x+v t) w(x) \extd x=\int_{\R} f(x) w (x-vt)\extd t. \]

As pointed out by one of the referees, the reader should note the similarities in the construction here with L\'evy processes on bialgebras~\cite{Uwe-Franz}.
		
\looseness=-1 The reader may ask why we continue to use an exponential with parameter in $ \R $ in a noncommutative setting. The differential setting of the KSGNS construction is very general, and could be used with other Hopf algebras replacing $C^{\infty}(\R)$. However in \cite{BegMa:geo} it is shown that using $ C^{\infty}(\R) $ is sufficient to describe quantum mechanics (the Schr\"odinger and Klein--Gordon equations) as auto parallel paths using the proper time as parameter. This shows that the $ C^{\infty}(\R) $ parameter case is of interest, thought not the most general case. The importance of paths on $ C^{*} $-algebras parametrised by the reals is illustrated by the definition of suspension of an algebra.

\section{Preliminaries}\label{Prelim}
A first-order differential calculus $ \Omega_{B}^{1} $ on an algebra $ B $ is a $ B$-bimodule with a derivation \linebreak \mbox{$ \extd \colon B \rightarrow \Omega_{B}^{1} $}, and so that $ \Omega_{B}^{1} $ is spanned by $ c \extd b $ where $ b, c \in B$. For a $ * $-algebra $B$, this will be a~$*$-differential calculus if there is an antilinear map $ *\colon \Omega_{B}^{1} \rightarrow \Omega_{B}^{1}$ so that
\[(c. \extd b)^{*}=\extd (b^{*}). c^{*}.\]
The right vector fields $ \chi_{B}^{R} $ consist of right module maps from $ \Omega_{B}^{1} $ to $ B $, with evaluation
\[ \ev\colon \ \chi_{B}^{R} \tens \Omega_{B}^{1} \rightarrow B.
\] For a left $ B $-module $ M $ a left connection is a linear map $ \nabla_{M}\colon M \rightarrow \Omega_{B}^{1} \tens_{B} M$ with the left Leibniz rule for $ b\in B $ and $ m \in M $
\begin{align}\label{e73}
	\nabla_{M}(b. m)=\extd b \tens m+b. \nabla_{M}(m).
\end{align}
In the case where $ M$ is a $ B$-$A $-bimodule we have a left bimodule connection $ (M, \nabla_{M}, \delta_{M}) $ when there is a bimodule map
\[\sigma_{M}\colon \ M \tens_{A} \Omega_{A}^{1} \rightarrow \Omega_{B}^{1} \tens_{B} M\]
for which we have the modified right Leibniz rule for $ a\in A $
\[\nabla_{M}(m. a)=\nabla_{M}(m). a +\sigma_{M}(m \tens \extd a). \]
Bimodule connections were introduced in \cite{DuboisMasson,DuboisMichor,Mourad} and extensively used in \cite{FioreMadore, Madore}.

For a Hopf algebra $ H $ we use the Sweedler notation $ \Delta h =h_{(1)}\tens h_{(2)}$. A differential calculus is called left covariant if there is a left $ H $-coaction $ \Delta_{L}\colon \Omega_{H}^{1} \rightarrow H \tens \Omega_{H}^{1} $ where $ \Delta_{L} (h. \extd k)= h_{(1)} k_{(1)} \tens h_{(2)} \extd k_{(2)} $ for $ h,k \in H $ \cite{WoroCal}. Similarly to the Sweedler notation, for a left coaction write $ \Delta_{L} (\xi)=\xi_{[-1]}\tens \xi_{[0]}$ for $ \xi \in \Omega_{H}^{1} $. We call $\Lambda_{H}^{1} $ the vector space of left invariant forms (i.e., $ \xi $ such that $\Delta_{L} \xi= 1 \tens \xi$ ). We now suppose that $ H $ has an invertible antipode, required by our choice of right vector fields and left coactions. The left coaction on $ \chi^{R}_H $ is defined to make the evaluation $ \ev\colon \chi^{R} _H \tens_{H} \Omega_{H}^{1}\rightarrow H$ a left comodule map, and is given by, for $ X \in \chi^{R}_H $ and $ \eta \in \Omega_{H}^{1}$
\begin{gather*}
	X_{[-1]}\tens X_{[0]} (\eta)=X(\eta_{[0]})_{(1)} S^{-1}\big(\eta_{[-1]}\big) \tens X\big(\eta_{[0]}\big)_{(2)}.
\end{gather*}
\begin{Definition}[\cite{MajidBook}]\label{D1}
	Two Hopf algebras $ H $ and $ H' $ are dually paired if there is a map $ \ev \colon H' \tens H\allowbreak \rightarrow \C $ which obey, for all $ \alpha, \beta \in H' $ and $ h,k \in H $
	\begin{gather*}
	\ev \langle \alpha, h k \rangle =\ev \langle \alpha_{(1)}, h, \rangle \ev \langle \alpha_{(2)}, k \rangle, \qquad \ev \langle \alpha \beta, h \rangle =\ev \langle \alpha, h_{(1)} \rangle \ev \langle \beta, h_{(2)} \rangle,\\
	\ev \langle 1_{H'}, h \rangle = \epsilon _{H} (h), \qquad \ev \langle \alpha, 1_{H} \rangle = \epsilon _{H'} (h), \qquad \ev \langle S \alpha, h\rangle =\ev \langle \alpha, S h\rangle.
	\end{gather*}
They are a strictly dual pair if this pairing is nondegenerate.

If $ H $ is finite-dimensional the idea of dual is quite simply the linear dual. However for infinite-dimensional Hopf algebras we must take more care. Notably the Hopf algebras $ \C_{q}[SU_{2}] $ and the deformed enveloping algebra $ U_{q}({\mathfrak{su}}_{2}) $ are dually paired, but $ U_{q}({\mathfrak{su}}_{2}) $ is much smaller than the continuous dual vector space of the $ C^{*} $-algebra $ \C_{q}[SU_{2}] $.
\end{Definition}
\begin{Definition}A right integral $ \phi\colon H\rightarrow \C $ is a linear map such that $ \phi \big(h_{(1)}\big) h_{(2)}=1_{H}. \phi (h)$. It is said to be normalised if $ \phi (1_{H}) =1$.
\end{Definition}
\begin{Definition}[\cite{MajidBook}] A Hopf algebra $ H $ which is also a $ * $-algebra is called a Hopf $ * $-algebra if
\begin{gather*}
	\Delta(h^{*})=h_{(1)}{}^{*}\tens h_{(2)}{}^{*},\qquad \epsilon(h^{*})=\epsilon(h)^{*}, \qquad (S\circ *)^{2}=\id.
\end{gather*}
\end{Definition}
For a Hopf $ * $-algebra we call a Haar right integral $ \phi $ Hermitian if $ \phi(h^{*}) =\phi(h)^{*}$.

\section{The KSGNS construction and paths}\label{Ch1}
The KSGNS construction \cite{Lance} for a completely positive map from $ C^{*} $-algebras $ A $ to $ B $ is given by an $ B $-$ A $ bimodule $ M $ and a Hermitian inner product
\begin{equation}\label{e72}
\langle\, ,\,\rangle\colon \ M \tens_{A} \overline{M} \longrightarrow B.
\end{equation}
Recall that the conjugate $ A$-$B $-bimodule $ \overline{M}$ is the conjugate $ \C $-vector space, with elements $ \overline{m}\in \overline{M} $ for $ m\in M $ and $ \overline{m+n}=\overline{m}+\overline{n} $ and $ \lambda \overline{m}=\overline{\overline{\lambda}m} $ for $ m,n \in M $ and $ \lambda \in \C $. The actions of the algebras are $ a. \overline{m}= \overline{m a^{*}} $ and $ \overline{m}. b= \overline{b^{*}m} $ for $ a \in A $ and $ b \in B $ \cite{BMBarcategories}. If we forget about completeness under a norm and positivity we can restate this in terms of more general $ * $-algebras. We shall take $ B= C^\infty(\R)$, and then in this case we just assume that $ \<\,,\,\> $ in~(\ref{e72}) is Hermitian, i.e., $ \<m,\overline{n}\>^{*} =\<n, \overline{m}\>$. We get $ \psi\colon A\rightarrow C^\infty(\R) $ given by
\begin{align}\label{e50}
\psi_t (a)=\langle m a, \overline{m}\rangle,
\end{align}
which is a time dependent linear functional, and in good cases a time dependent state.

We define the time evolution of $ \psi_t $ by imposing the condition $ \nabla _{M}m=0 $ on $ m $ in~(\ref{e50}) where~$ \nabla_{M} $ is a left $ B $-connection as in~(\ref{e73}). An obvious condition to place on the connection~$ \nabla_{M} $ is that it preserves the inner product, i.e., that
\begin{equation*}
\extd \langle m, \overline{n}\rangle= (\id \tens \langle\,,\,\rangle)\big(\nabla_{M} m \tens \overline{n}\big)+ (\langle\,,\,\rangle \tens \id)\big(m\tens \nabla_{\overline{M}} \overline{n}\big)
\end{equation*}
with $ \nabla_{\overline{M}} n=\overline{p} \tens \xi^{\ast}$ where $ \nabla_{M} n=\xi \tens p$. Note that this is just the usual preserving inner product condition used in Riemannian geometry \cite{BMbook}. As special case we consider $ M= C^{\infty}(\R) \tens A $ with actions given by product making it into a $ C^{\infty}(\R) $-$ A $ bimodule. We define the inner product~$ \langle\,,\,\rangle $ on~$ M $
\begin{align*}
	\< f_{1} \tens a_{1}, \overline{f_{2} \tens a_{2}} \>= f_{1} f_{2}^{*} \phi (a_{1}a_{2}^{*})
\end{align*}
for $ f_{i} \in C^{\infty }(\R) $ and $ a_{i} \in A $, where $ \phi\colon A\rightarrow \C $ is Hermitian map (i.e., $ \phi(a^{*})=\phi(a)^{*} $) and in nice cases a positive map. In terms of the $ A $ valued function of time approach, this is just $ \< m, \overline{n}\>(t)=\phi (m(t) n(t)^{*}) $ for $ m,n\in M $.

We consider the special case where $ \nabla_{M} $ is a bimodule connection. In~\cite{Bgeo} this is used to recover classical geodesics, but we use this assumption as it gives us a role for vector fields. It also would allow us to define a velocity for the paths, but we do not go into this.

We now take the $ C^{\infty}(\R) $-$ A $ bimodule $ C^{\infty}(\R) \tens A $ in the previous theory. However, we quickly find out that this bimodule will not in general contain the solution of the differential equations, and so pass to a larger bimodule $ C^{\infty}(\R, A) $, the infinitely differentiable functions from $ \R $ to~$ A $. Outside the case where $ A $ is finite-dimensional (and the two definitions are the same) we would require some topology to define differentiable, but our infinite-dimensional examples are $ C^{*} $-algebras.
\begin{Proposition}[\cite{Bgeo}] \label{prido}
	For a unital algebra $A$ with calculus $\Omega_A$ and $C^\infty(\R)$ with its usual calculus $\Omega(\R)$ we set $M=C^{\infty}(\R, A) $. Then a general left bimodule connection on $M$ is of the form, for $m\in C^\infty(\R)\tens A$ and $\xi\in\Omega^1_A$
	\[
	\nabla_M(m)=\extd t\tens \big(pm + \tfrac{\partial m}{\partial t} +X(\extd m)\big),\qquad \sigma_M(1\tens\xi)=\extd t\tens X(\xi)
	\]
	for some $p\in C^{\infty}(\R, A)$ and $X\in C^{\infty}\big(\R, \chi ^R\big)$ where $ \extd$ is the derivation $ \extd\colon A\rightarrow \Omega_{A}^{1} $. $\big($Note that explicitly including time evaluation we have
	$X(\eta)(t)=X(t)(\eta(t))$ for $\eta\in C^{\infty}\big(\R, \Omega^1_A\big).\big)$ Further the connection preserves the inner product on $M$ if for all $a\in A$ and $\xi\in\Omega^1_A$.
	\begin{gather}\label{e67dc}
			\<\big(pa+X(\extd a)+ap^*
		\big),\overline{1}\>=0 =\<X(\xi^*)-X(\xi)^*,\overline{1}\>.
	\end{gather}
\end{Proposition}

Following from the classical theory, we shall call the first equality of the equation~(\ref{e67dc}) the divergence condition for $p$ and the second the reality condition for $X$. In this case the divergence $ \operatorname{div} (X) \in A $ of $ X \in \chi^{R} $ for all $ a \in A $ is given by
\begin{align*}
\phi (\operatorname{div} (X).a+X(\extd a))=0.
\end{align*}
In \cite{Bgeo} it is shown that we can set $ p= \frac{1}{2} \operatorname{div} (X) $ in Proposition~\ref{prido}.

In this paper we only consider the case of a Hopf algebra $ H $ and a left invariant right vector field~$ X $. Now if both $ X \in \chi^{R} $ and $ \xi \in \Omega_{H}^{1} $ are left invariant we find that $ X(\xi) \in H$ is left invariant, so it is a multiple of the identity. We use the invariant derivative $ \omega \colon H\rightarrow \Lambda_{H}^{1} $ which is defined so that $ \extd h=\omega \big(h_{(2)}\big). h_{(1)} $,
\begin{gather}\label{e58}
	\omega (a)=\extd a_{(2)} S^{-1}\big(a_{(1)}\big) \in \Lambda_{H}^{1}.
\end{gather}
\begin{Proposition}\label{prop2}
	If $ \phi $ is a Hermitian right Haar integral on $ H $ and $ X \in \chi^{R} $ is left invariant, then $ \operatorname{div} (X) =0 $.
\end{Proposition}
\begin{proof}
	As $ X \big( \omega (a_{(2)}) \big) $ is just a number in the following expression
	\begin{gather*}
	X(\extd a) =X\big( \extd a_{(3)} S^{-1} \big(a_{(2)}\big) a_{(1)} \big)= X \big( \omega \big(a_{(2)}\big) \big) a_{(1)},
	\end{gather*}
	so $ \phi (	X(\extd a))	=X \big( \omega \big(a_{(2)}\big) \big) \phi \big(a_{(1)} \big) $. For the Hermitian right Haar integral $ \phi \big(a_{(1)}\big) a_{(2)}=\phi (a).1 $
	so $ 	\phi (	X(\extd a))	= X( \omega( 1 ) )\phi (a)=0$.
\end{proof}
\begin{Proposition}\label{P3}
	For a Hopf $ * $-algebra $ H $ with a left invariant $ * $-calculus and $ \phi $ is a Hermitian right Haar integral, to show that a left invariant right vector field $ X \in\chi^{R} $ is real, it is sufficient to check that $ X(\eta^{*})= X(\eta)^{*} $ for all $\ \eta \in \Lambda_{H}^{1} $.
\end{Proposition}
\begin{proof}
	Recalling the property $ \phi (a_{(1)}) a_{(2)}=\phi (a).1$ for a right Haar integral, we have for $ \xi \in \Omega^{1} $
	\begin{gather*}
		1. \phi (X(\xi ^{*}))= \phi \big(X(\xi ^{*})_{(1)}\big) X(\xi ^{*})_{(2)},\\
			1. \phi (X(\xi )^{*})= \phi \big(X(\xi )^{*}{}_{(1)}\big) X(\xi )^{*}{}_{(2)}= \phi \big(X(\xi )_{(1)}{}^{*}\big) X(\xi )_{(2)}{}^{*},
	\end{gather*}
and as $ X $ is left invariant and $ \ev\colon \chi^{R} \tens \Omega^{1}\rightarrow H$ is a left comodule map we have
\begin{gather*}
	1. \phi (X(\xi ^{*}))= \phi \big(\xi ^{*}{}_{[-1]}\big) X\big(\xi ^{*}{}_{[0]}\big)	= \phi \big(\xi _{[-1]}{}^{*}\big) X\big(\xi _{[0]}{}^{*}\big),\\
		1. \phi (X(\xi )^{*})= \phi \big(\xi _{[-1]}{}^{*}\big) X\big(\xi _{[0]}\big)^{*}.
\end{gather*}
For $ h \in H $, $ \phi(h) $ is in the field and $ \phi(h^{*}) =\phi(h)^{*}$ and so
\begin{gather*}
	1. \phi (X(\xi ^{*}))= X\big(\phi \big(\xi _{[-1]}\big)^{*}\xi _{[0]}{}^{*}\big)= X\big(\big(\phi \big(\xi _{[-1]}\big)\xi _{[0]}\big)^{*}\big), \\
	1. \phi (X(\xi )^{*})= X\big(\phi \big(\xi _{[-1]}\big)\xi _{[0]}\big)^{*}.
\end{gather*}
Finally note that $ \eta=\phi\big(\xi_{[-1]}\big) \xi_{[0]}$ is left invariant.
\end{proof}
\begin{Theorem}\label{prp2tothm}
		The connection $ \nabla_{M} m= \extd t \tens (\dot{m}+X (\extd m )) $ for left invariant $ X\in \chi^{R} $ has solutions of $ \nabla_{M} m=0 $ given by
	\[m(t)=m(0)_{(1)} \exp(-t( X\circ \omega))\big(m(0)_{2}\big),\]
	where we take the exponential as a power series in elements of $H'$.
\end{Theorem}
\begin{proof}
	We solve $ \dot{m}=-X (\extd m ) $ by using $ \extd m= \omega (m_{(2)}) m_{(1)}$, so
	\begin{gather}\label{e42}
	\dot{m}=- (X \circ\omega )\big(m_{(2)}\big) m_{(1)},\qquad	\ddot{m}=- ( X\circ\omega )\big(\dot{m}_{(2)}\big) m_{(1)} - ( X\circ\omega ) \big(m_{(2)}\big) \dot{m}_{(1)}.
	\end{gather}
	As $ \Delta $ and $ 	\frac{\extd}{\extd t} $ on $ M $ commute
	\begin{gather*}
	\frac{\extd}{\extd t} \big(m_{(1)} \tens m_{(2)}\big) =\dot {m}_{(1)} \tens m_{(2)}+ m_{(1)} \tens \dot {m}_{(2)}=- m_{(1)}\tens m_{(2)} ( X\circ\omega )\big(m_{(3)}\big),
	\end{gather*}
	and substituting this back into (\ref{e42}) gives
	\begin{gather*}
	\ddot{m}=( X\circ\omega ) \big(m_{(3)}\big) ( X\circ\omega ) \big(m_{(2)}\big) m_{(1)} =( X\circ\omega )^{2}\big(m_{(2)}\big) m_{(1)}
	\end{gather*}
using the product in $ H'$. Continuing with higher derivatives and using Taylor's theorem to get the answer, recalling that the first term in the exponential, the identity in $ H' $, is $ \epsilon $.
\end{proof}

We can use this formula for $ m(t) $ in $ \psi_t(a)=\phi(m(t) a m(t)^{*}) $ to give
\begin{gather}\label{e68}
	\psi_t(a)=\exp (-tX \circ \omega) \big(m(0)_{(2)}\big) \exp (-tX \circ \omega) \big(n(0)_{(2)}\big) ^{*} \phi \big(m(0)_{(1)} a n (0)_{(1)}^{*}\big),
\end{gather}
where $ n $ is an independent copy of $m$. Note that for a classical geodesic on a group starting at the identity element we would have $ H= C^{\infty}(G) $ and $ m(0) $ would be a $ \delta $-function (or more accurately $ \frac{1}{2} $ density) at the identity $ e\in G $, giving
\begin{gather*}
	\psi_0= \epsilon\colon \ C^{\infty}(G) \rightarrow \R.
\end{gather*}

 \section[Functions on a finite group G]{Functions on a finite group $\boldsymbol{G}$} \label{S1}

 We take $ H=\C[G] $, the functions on a finite group $ G $. A basis is $ \delta_{g} $ for $ g\in G $, the function taking value $ 1 $ at $ g $ and zero elsewhere. This is a Hopf algebra with
 \[\epsilon (\delta_{g})=\delta_{g,e}, \qquad \Delta \delta_{g} = \sum_{ x,y\in G \colon xy=g}\delta_{x}\tens \delta_{y}, \qquad S(\delta_{g})=\delta_{g^{-1}}.\]
 The first-order left covariant differential calculi on $ H\!=\!\C[G]$, correspond to subsets $\mathcal{C}\!\subseteq\! G\setminus\{e\}$~\cite{KB}. The basis as a left module for the left invariant 1-forms is~$e^a$ for $a\in\mathcal{C}$, with relations and exterior derivative for $f\in \C[G]$ being
 \[ e^a.f=R_a(f) e^a,\qquad \extd
 f=\sum_{a\in {\mathcal{C}}}(R_a(f)-f) e^a,\]
 where $R_a(f)(g)=f(g a)$ denotes right-translation. We take $ e_{b} $ for $ b \in \mathcal{C} $ to be the dual basis to $ e^{a} \in \Lambda^{1} $, i.e., $ \ev (e_{b}\tens e^{a} )=e_{b} (e^{a})= \delta_{a,b} $. Now from~(\ref{e58})
\begin{equation*}
\omega(\delta_{g})=\begin{cases}
\sum e^{a} & \mbox{if} \ g=e, \\
- e^{a} & \mbox{if}\ g^{-1}=a \in \mathcal{C},\\
0 & \mbox{otherwise},
\end{cases}
\end{equation*}
so if we set $ X=\sum X^{a} e_{a} \in \textgoth{h}=(\Lambda_{H}^{1})'$, for some $ X^{a}\in H $, then{\samepage
 \begin{equation}\label{e51}
(X \circ \omega)(\delta_{g})=\begin{cases}
\sum X^{a} & \mbox{if} \ g=e, \\
- X^{a} & \mbox{if}\ g^{-1}=a \in \mathcal{C},\\
0 & \mbox{otherwise}.
\end{cases}
\end{equation}
 We set $ \phi $ to be the normalised Haar measure $ \phi(f)=\frac{1}{|G|}\sum_{g\in G}f (g) $.}

Now $ (X \circ \omega) $ is an element of the dual of $ H=\C(G) $, which is the group algebra $H'= \C G $. To write elements of the dual Hopf algebra we first list the elements of $ G $ as $ g_{1},g_{2}, \dots, g_{n} $ and then for $ \beta \in H' =\C G $ we use a column vector notation
\begin{equation}\label{e52}
	\beta= \left( \begin{matrix}
		\beta(\delta_{g_{1}})\\
		\beta(\delta_{g_{2}})\\
		\vdots\\
		\beta(\delta_{g_{n}})
	\end{matrix} \right).
\end{equation}
It will be convenient to turn the calculation of the exponential on $ \C G $ into a matrix exponential using a differential equation. We set $ \alpha_t =\exp (t X \circ \omega) $ so $ \frac{ \extd \alpha_t }{ \extd t}= \alpha_t. (X \circ \omega) $ and
\begin{equation}\label{e12}
\dfrac{ \extd \alpha _t(\delta_{g}) }{ \extd t}= \sum_{ xy=g} \alpha_t (\delta_x) (X \circ \omega(\delta_{y})).
\end{equation}
Now we can write (\ref{e12}) as matrix equation $ \dot{\alpha_t }=T\alpha _t $
\begin{equation}\label{e46}
\dot{\alpha_t }(\delta_{g_{i}})=\sum_{k} \alpha_t (\delta_{g_{k}}). (X \circ \omega)\big(\delta_{g_{k}^{-1} g_{i}}\big),
\end{equation}
where
\begin{gather}\label{e13}
T_{ik}= (X \circ \omega)\big(\delta_{g_{k}^{-1} g_{i}}\big)=\begin{cases}
\sum X^{a} & \mbox{if} \ k=i, \\
- X^{a} & \mbox{if}\ g_{i}^{-1} g_{k}=a \in \mathcal{C},\\
0 & \mbox{otherwise}.
\end{cases}
\end{gather}
Now we have $ \alpha_t= \exp (t T) \beta $ where $ \beta $ is the column vector corresponding to the identity in $ \C G $ and we use the matrix exponential.
The calculus on $ H $ has a $ * $-structure given by $ {\rm e}^{a^{*}}=-{\rm e}^{a^{-1}} $, so by Proposition \ref{P3} the left invariant vector field $ X=\sum X^{a} e_{a} $ is real if $ X(^{a} )^{*}=-X^{a^{-1}}$.
\begin{Example}
	Let $ G = S_{3} $, set $ a \in \mathcal{C}=\{u,v,w \} $ where $ u= (1,2 ) $, $ v= (2,3 )$ and $ w= (1,3 )$, and write $ 	X=X^{u}e_{u}+X^{v} e_{v}+ X^{w} e_{w} $	where $ X^{u},X^{v},X^{w}\in \C $ and the elements of $ S_{3} $ are listed as
	\begin{gather}\label{e57}
	g_{1}=e,\qquad g_{2}= 	 (
	1, 2, 3
	 ), \qquad g_{3}= 	 (
	1, 3, 2
	 ), \qquad g_{4}=u, \qquad g_{5}=v, \qquad g_{6}=w.
	\end{gather}
	 Now the matrix $ T $ in (\ref{e13}) becomes
	\begin{equation*}
	T=	\begin{pmatrix}
	T_{11} & 0 & 0 & -X^{u}& -X^{v}& -X^{w}\\
	0 & T_{22} & 0 & -X^{v}& -X^{w}& - X^{u}\\
	0 & 0 & T_{33} &-X^{w}& -X^{u}& -X^{v}\\
	-X^{u}& - X^{v}& - X^{w} &T_{44} & 0 & 0\\
	-	X^{v}& -X^{w}& -X^{u} & 0 & T_{55} & 0\\
	-	X^{w}& -X^{u}& -X^{v} & 0 &0 & T_{66}
	\end{pmatrix},
	\end{equation*}
	where the diagonal elements of $ T$ are $ T_{ii}=X^{u}+X^{v}+X^{w}$. Now the solution to $ \dot{\alpha_t}=T\alpha_t $ is $ \alpha _t=\exp (tT) \alpha _0$. Set $ \alpha_0 =e $, the identity in $ S_{3} $, which is the column vector $ \left(\begin{array}{cccccc}
	1	&0 &0 &0 &0 &0
	\end{array}
	\right) ^{\rm T} $, and time $ t=1 $. We set $ X^{u}={\rm i}p $, $ X^{v}={\rm i}q $ and $X^{w}={\rm i}r $ and $ \gamma=\sqrt{p^2+q^{2}+r^{2}-pq-pr-qr} $ for $ p,q,r \in \C $ to get
\begin{gather}
	\exp ({\rm i} ( pe_{u}+qe_{v}+ r e_{w})\circ \omega)\nonumber\\
\qquad{}= \frac{1}{3}{\rm e}^{{\rm i}(p+q+r)} \left(
	\begin{matrix}
		2 \cos (\gamma )+\cos (p+q+r) \\
		\cos (p+q+r)-\cos (\gamma ) \\
		\cos (p+q+r)-\cos (\gamma ) \\
		-{\rm i} \left(\dfrac{\sin (\gamma ) (2 p-q-r)}{\gamma}+\sin (p+q+r)\right) \vspace{1mm}\\
			-{\rm i} \left( \dfrac{{\rm i} \sin (\gamma ) (2 q-p-r)}{\gamma}+ \sin (p+q+r) \right) \vspace{1mm}\\
			-{\rm i}\left( \dfrac{{\rm i} \sin (\gamma ) (2 r-p-q)}{\gamma}+\sin (p+q+r) \right)
	\end{matrix}
	\right)\in \C S_{3}.\label{e45}
\end{gather}
In our case $ \mathcal{C}=\{u,v,w\} $ so $ a^{-1}=a $ for $ a\in \mathcal{C} $, and the reality condition is that $ p,q,r\in \R $ and as a result $ \gamma\in \R $. Note that the vector does not depend on the sign of the square root and that the $ L^{2} $ norm of the vector in~(\ref{e45}) is equal to $ 1 $. Now we look at time dependence of the state $ \psi_t $ given by equations $ \psi_t(a)=\phi(m(t) a m(t)^{*}) $ and~(\ref{e68}). We start with $ m(0) =\delta_{e}$ at $ t=0 $ where everything is concentrated at the identity. We have $ \Delta m (0)=\Delta \delta_{e}=\sum_{g } \delta_{g^{-1}} \tens \delta_{g} $, so~(\ref{e68}) gives
\begin{align}
\psi_t(a)&=\sum_{g,h}\phi (\delta_{g^{-1}} a \delta_{h^{-1}} )\langle \exp (-X \circ \omega), \delta_{g} \rangle \langle \exp (-X \circ \omega), \delta_{h} \rangle^{*}\nonumber\\
&=\sum_{g }\phi (\delta_{g^{-1}} a )\left| \langle \exp (-X \circ \omega), \delta_{g} \rangle \right|^{2},\label{e53}
\end{align}
so $ \phi $ is given by a probability density $ \sum_{g } \delta_{g^{-1}} | \langle \exp (-X \circ \omega), \delta_{g} \rangle |^{2} $. In terms of the group algebra, which is dual to the functions,
\begin{gather*}
9 \psi_t = e \left( 2 \cos (\gamma)+\cos (p+q+r)\right) ^{2}+ \big((1 3 2)+ (1 2 3 )\big)\left( \cos (p+q+r)-\cos (\gamma) \right)^{2}\cr
\hphantom{9 \psi_t =}{}
+u \left( \frac{(2 p-q-r) \sin (\gamma)}{\gamma}+\sin (p+q+r) \right)^{2}\\
\hphantom{9 \psi_t =}{}
+v \left( \frac{(2 q-p-r) \sin (\gamma)}{\gamma}+\sin (q+p+r) \right)^{2} \cr
\hphantom{9 \psi_t =}{} +w \left(\frac{(2 r-q-p) \sin (\gamma)}{\gamma}+\sin (r+q+p) \right) ^{2}.
\end{gather*}
To plot some example exponential of states we refer back to the ordering of group elements in~(\ref{e57}), and plot the weight of each element against time for $ 0\leq t\leq 7 $. We display some cases in Figure~\ref{F1}. This illustrates the conversion of the solution in $ (p,q,r) $ in particular cases to a~function of the parameter~$ t$. Note in general the exponential map will not be periodic as the ratio between $ \gamma $ and $ p+q+r $ is likely not to be rational.
\end{Example}

\begin{figure}[h!]	\centering
	\begin{subfigure}[b]{0.32\textwidth}
		\includegraphics[width=1\textwidth]{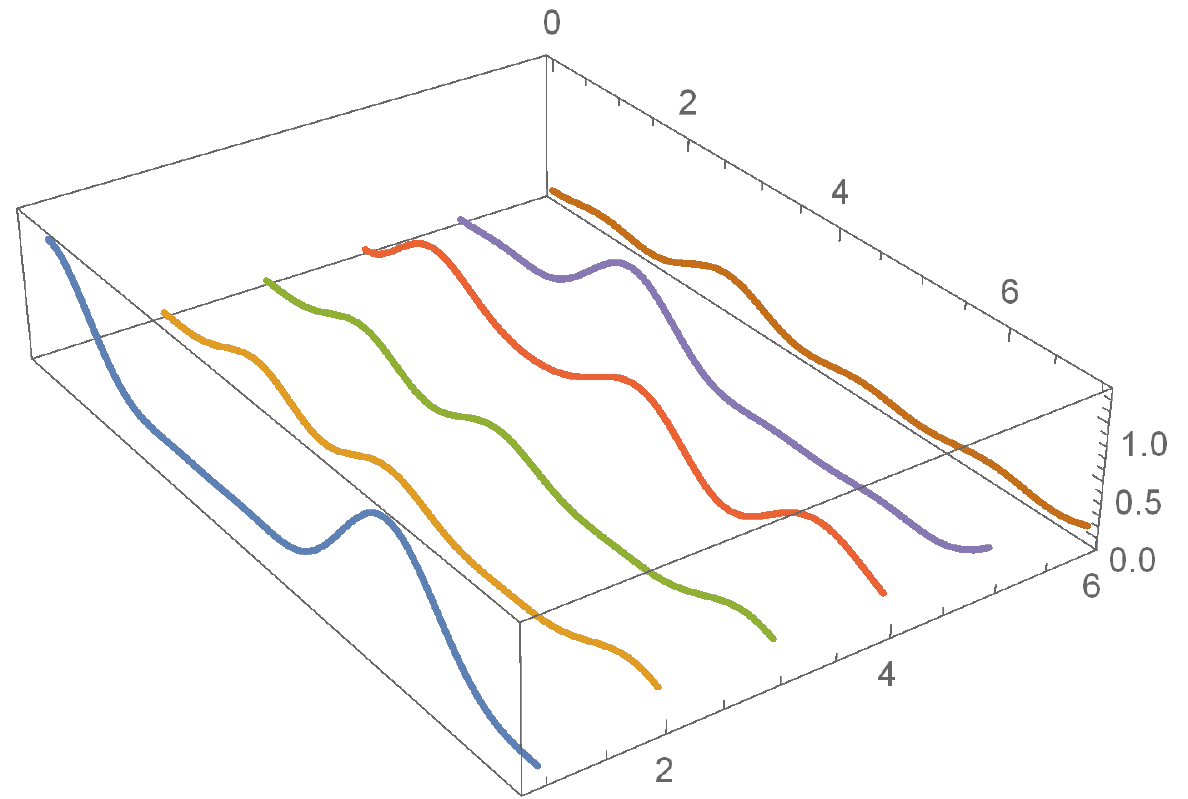}
	\end{subfigure}
	\begin{subfigure}[b]{0.32\textwidth}
		\includegraphics[width=1\textwidth]{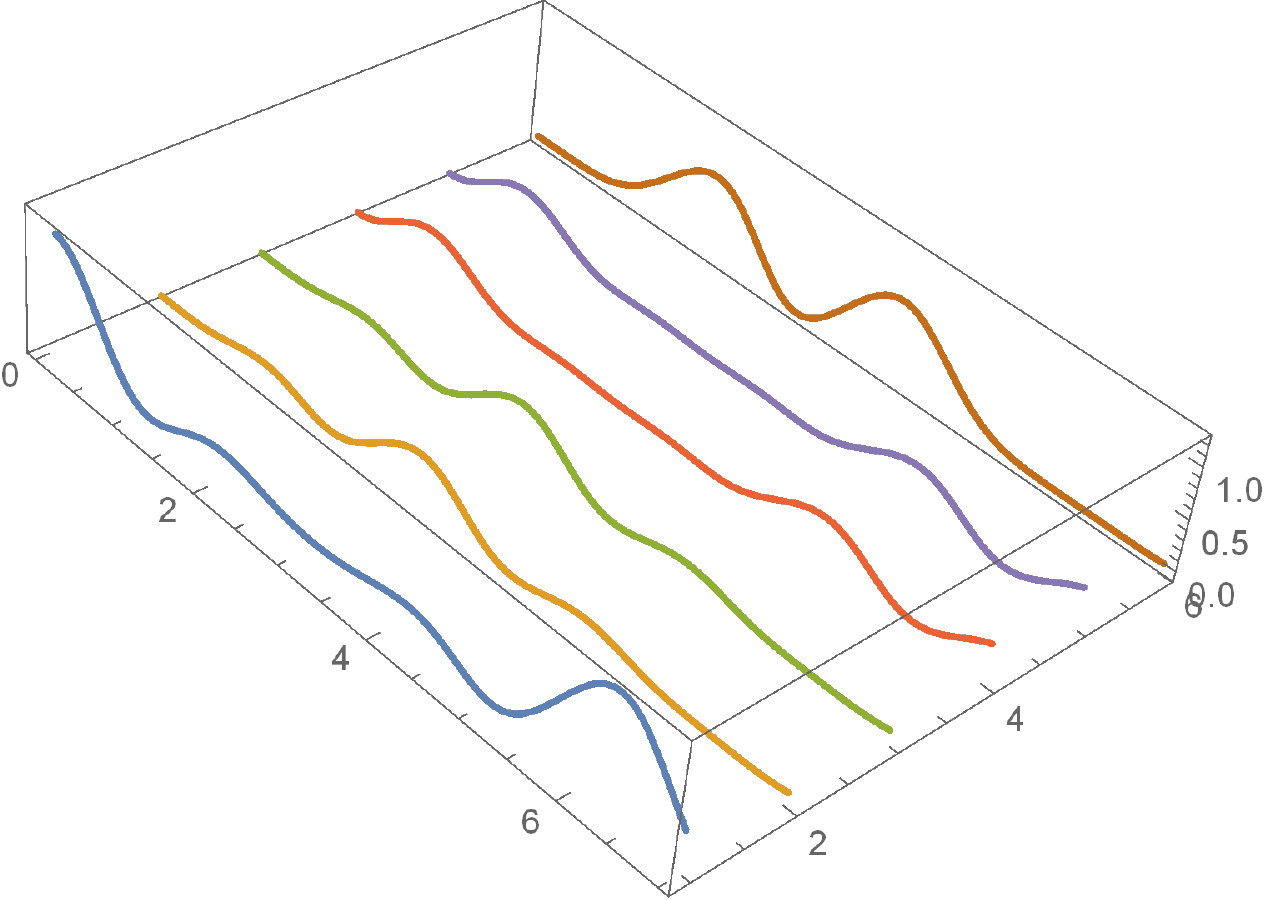}
	\end{subfigure}
	\begin{subfigure}[b]{0.32\textwidth}
		\includegraphics[width=1\textwidth]{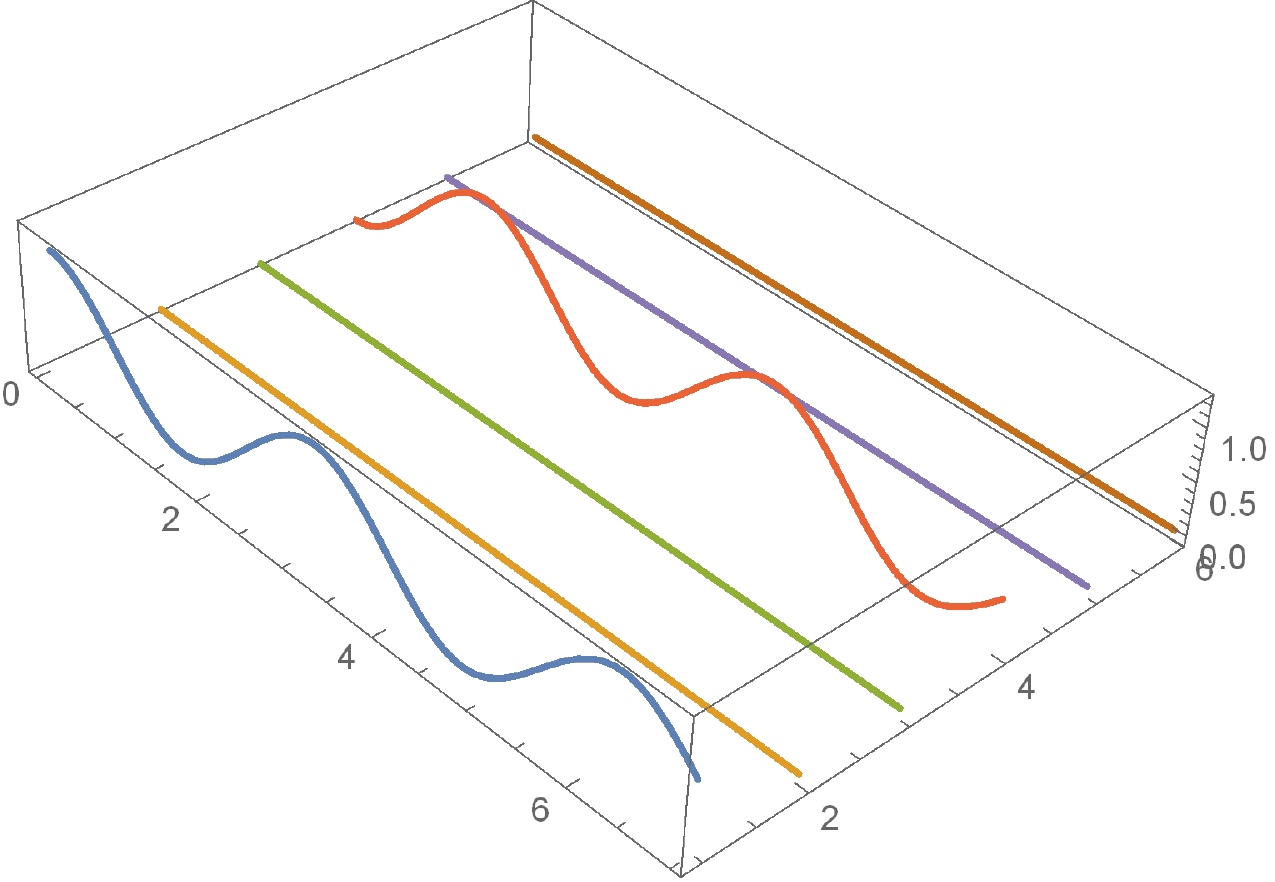}
	\end{subfigure}
	\caption{The states $ \psi_t $ given by exponentials for vector fields $ X={\rm i} t e_{u}+\frac{1}{3}{\rm i} t e_{v}+\frac{1}{2}{\rm i} te_{w} $ and $ X={\rm i} t e_{u}+{\rm i} t e_{v}$ and $ X={\rm i} t e_{u}$ respectively.}\label{F1}
\end{figure}

\section[Functions on the integers Z]{Functions on the integers $\boldsymbol{\Z}$}

We shall apply the finite group methods of Section~\ref{S1} to the group $ \Z $, which needs to be treated with care. We shall use rapidly decreasing functions and an un-normalised Haar measure $ \phi(f) =\sum_{n\in \Z} f(n)$ and infinite matrices. The column vector notation of~(\ref{e52}) becomes, truncating the infinite vectors
\begin{gather}\label{e69}
\beta= \left( \begin{matrix}
\beta(\delta_{2})\\
\beta(\delta_{1})\\
\beta(\delta_{0})\\
\beta(\delta_{-1})\\
\beta(\delta_{-2})
\end{matrix} \right), \qquad {z_{-1}}=\left( \begin{matrix}
0\\
0\\
0\\
1\\
0
\end{matrix}
\right),\qquad {z_0}=\left( \begin{matrix}
0\\
0\\
1\\
0\\
0
\end{matrix}
\right),\qquad {z_1}=\left( \begin{matrix}
0\\
1\\
0\\
0\\
0
\end{matrix}
\right).
\end{gather}
We have used the pairing of the group algebra $ \C \Z $ with the functions to give the vectors corresponding to $ Z_{n}\in \C \Z $, the basis elements corresponding to $ n\in \Z $.
We look at equation~(\ref{e46}) in the case of the integers, which becomes
\begin{align}\label{e63}
\dot{\alpha_t}(\delta_{i})=\sum_{k}\alpha_t (\delta_{k}) (X \circ \omega)(\delta_{i-k}).
\end{align}
We use $ \mathcal{C}=\{+1,-1\} $ for the calculus, giving two generators $ e^{+1}$, $e^{-1} $ with dual invariant vector fields $ e_{+1}$ and $ e_{-1} $. Now for the vector field $ X =X^{+1}e_{+1}+X^{-1}e_{-1}$ (\ref{e51}) becomes
\begin{gather*}
(X \circ \omega)(\delta_{g})=\begin{cases}
X^{+1} +X^{-1} & \mbox{if} \ g=0, \\
- X^{-1} & \mbox{if}\ g=1,\\
- X^{+1} & \mbox{if}\ g=-1,\\
0 & \mbox{otherwise},
\end{cases}
\end{gather*}
so (\ref{e63}) becomes
\begin{gather}\label{e55}
\dot{\alpha_t}(\delta_{i})=-\alpha_t (\delta_{i+1}) X^{+1} -\alpha _t(\delta_{i-1}) X^{-1}+\big(X^{+1} +X^{-1} \big) \alpha_t (\delta_{i}).
\end{gather}
To describe this more easily we use matrices $ N_{n} $ (infinite in both directions, we only consider a~part centred on the $ 0$, $0 $ entry)
\begin{gather*}
	N_{1}=\left( \begin{matrix}
		0& 1& 0 & 0 & 0 \\
		0& 0 & 1& 0 & 0 \\
	0	& 0 & 0 &1 & 0 \\
		0& 0 & 0&0 &1 \\
		0& 0& 0& 0 & 0
	\end{matrix}
	\right), \qquad N_{-1}=\left( \begin{matrix}
		0& 0 &0 & 0& 0 \\
		1& 0 & 0 & 0 & 0\\
		0& 1 & 0 & 0& 0\\
		0& 0 & 1&0 & 0 \\
		0& 0& 0&1 & 0
	\end{matrix}
	\right), \qquad N_{2}=\left( \begin{matrix}
		0& 0 & 1& 0 & 0 \\
		0& 0 & 0 & 1& 0 \\
		0& 0& 0 & 0 & 1 \\
		0& 0& 0 &0 & 0 \\
	0	& 0& 0&0 & 0
	\end{matrix}
	\right),
\end{gather*}
etc., and $ N_{n}N_{m}=N_{n+m} $. If we write $ \alpha_t $ as a column vector similarly to~(\ref{e69}), then we can write the differential equation
\[\dot{\alpha_t}= \big({-}X^{+1} N_{-1}+\big(X^{+1}+X^{-1}\big) N_{0}-X^{-1}N_{1}\big)\alpha_t.\]
 To find the exponential we need to use a generalised hypergeometric function~\cite{web}
	\begin{gather*}
{}_{0}F_{1}(;a ;x)=1+\dfrac{x}{a 1!}+\dfrac{x^{2}}{a (a+1) 2!}+\cdots.
\end{gather*}
\begin{Proposition}\label{P2}
	\begin{gather*}
\begin{split}
&	\exp(-aN_{1}+(a+b)N_{0}-bN_{-1})\\
&\qquad{} ={\rm e}^{a+b}\bigg({}_{0}F_{1}(;1;ab)N_{0}+\sum_{n > 0}\dfrac{{}_{0}F_{1} (;n+1;ab)}{n!}\big((-a)^{n}N_{n}+(-b)^{n}N_{-n}\big)\bigg),
\end{split}
	\end{gather*}
so
	\begin{gather*}
\alpha _{t} =\exp (tX \circ \omega)	=\exp\big({-}t X^{-1}z_{1}+\big(t X^{-1}+t X^{+1}\big)z_{0}-t X^{+1}z_{-1}\big)\cr
\hphantom{\alpha _{t}}{}
={\rm e}^{t(X^{+1}+X^{-1})}\bigg({}_{0}F_{1}\big(;1;t^{2} X^{-1}X^{+1}\big)z_{0}\\
\hphantom{\alpha _{t} =}{}
+\sum_{n > 0}\dfrac{{}_{0}F_{1} \big(;n+1;t^{2} X^{-1}X^{+1}\big)}{n!}\big(\big({-}t X^{-1}\big)^{n}z_{n}+(-t X^{+1})^{n}z_{-n}\big)\bigg).
\end{gather*}
\end{Proposition}
\begin{proof}
	Using the trinomial theorem
	\begin{align*}
	(-a N_{1}+(a+b)N_{0}-bN_{-1})^{n}&=\sum_{i,j,k\geq 0:i+j+k=n} \frac{n!}{i! j! k!} (-a N_{1})^{i}((a+b)N_{0})^{j}(-bN_{-1})^{k}\\
	&=\sum_{i,j,k\geq 0:i+j+k=n} \frac{n!}{i! j! k!} (-a)^{i}(a+b)^{j}(-b)^{k} N_{i-k},
	\end{align*}
	so
	\begin{align}
	\exp (-a N_{1}+(a+b)N_{0}-bN_{-1})&=\sum_{i,j,k\geq 0}\frac{(-a)^{i}(a+b)^{j}(-b)^{k} }{i! j! k!} N_{i-k}\nonumber\\
	&={\rm e}^{a+b}\sum_{i,k\geq 0}\frac{(-a)^{i}(-b)^{k} }{i! k!}N_{i-k}.\label{e47}
	\end{align}
	If we set $ i-k =n$ we get sums depending on the sign of $ n $. For $ n \geq 0 $ we have the coefficient of $ N_{n} $ in~(\ref{e47}) being
	\begin{gather*}
	{\rm e}^{a+b}\sum_{k\geq 0}\frac{(-a)^{k+n}(-b)^{k} }{(k+n)! k!}=	{\rm e}^{a+b}\frac{(-a)^{n}}{n!}\sum_{k\geq 0}\frac{n! }{(k+n)! k!}(a b)^{k}=	{\rm e}^{a+b}\frac{(-a)^{n}}{n!}{}_{0}F_{1}(;n+1 ;ab).
	\end{gather*}
	For $ n< 0 $ we have (putting $ m=-n $) $ k=i+m $ and the coefficient of $ N_{-m} $ in~(\ref{e47}) is
	\begin{gather*}
	{\rm e}^{a+b}\sum_{i\geq 0}\frac{(-a)^{i}(-b)^{i+m} }{i! (i+m)! }=	{\rm e}^{a+b}\frac{(-a)^{m}}{m!}\sum_{i\geq 0}\frac{m! }{(i+m)! i!}(a b)^{i}= {\rm e}^{a+b}\frac{(-b)^{m}}{n!}{}_{0}F_{1}(;1+m ;ab).
	\end{gather*}
To calculate $ \exp (t X \circ \omega) $ we put $ a=tX^{-1} $ and $ b=tX^{+1} $ and apply the matrix exponential to $ z_{0}\in \C \Z $ to find the answer.
\end{proof}

Now we follow the equation (\ref{e53}) in finding the state corresponding to the initial state which is the dual element $ z_{0} $, corresponding to $ m(0)=\delta_{0} $,
\begin{gather*}
\psi_t(a)=\phi(m(t) a m(t)^{*})=\sum_{r\in \Z}\phi (\delta_{r} a) \left| \exp(-tX\circ \omega)(\delta_{-r}) \right| ^{2}.
\end{gather*}
As $ \phi (\delta_{r}a)=a(r)$ we see that the state is the element of the dual $ \sum_{n} z _{n } | \exp (-tX \circ \omega) (\delta_{-n}) |^{2} $. Now we restrict to the real vector field case $ \big(X^{+1} \big)^{*}=-X^{-1}$, where $ X^{+1}+X^{-1} $ is imaginary, so $ \big|{\rm e}^{t(X^{+}+X^{-1})}\big|=1 $, and $ \big|X^{+1}\big|^{2}=-X^{+1}X^{-1}>0 $. To plot a graph we take the case $
\big|X^{+1}\big| =1$
\begin{gather*}
\psi_t=z_{0}\big| {}_{0}F_{1}\big(;1;-t^{2}\big)\big| ^{2}+ \sum_{n \geq 1}(z_{n }+z_{-n })\left|\frac {{}_{0}F_{1} \big(;n+1;-t^{2}\big)}{n!}\right| ^{2}|t| ^{2n}.
\end{gather*}
We plot this for integers $ -4 \leq n \leq 4 $ in the range $ 0\leq t \leq 5 $ in the first graph in Figure~\ref{E2}, plotted using standard functions in Mathematica, and there it can be seen that there is a~damped oscillatory behaviour.
\begin{figure}[t]
	\centering
	\begin{subfigure}[b]{0.32\textwidth}
		\includegraphics[width=0.7\textwidth]{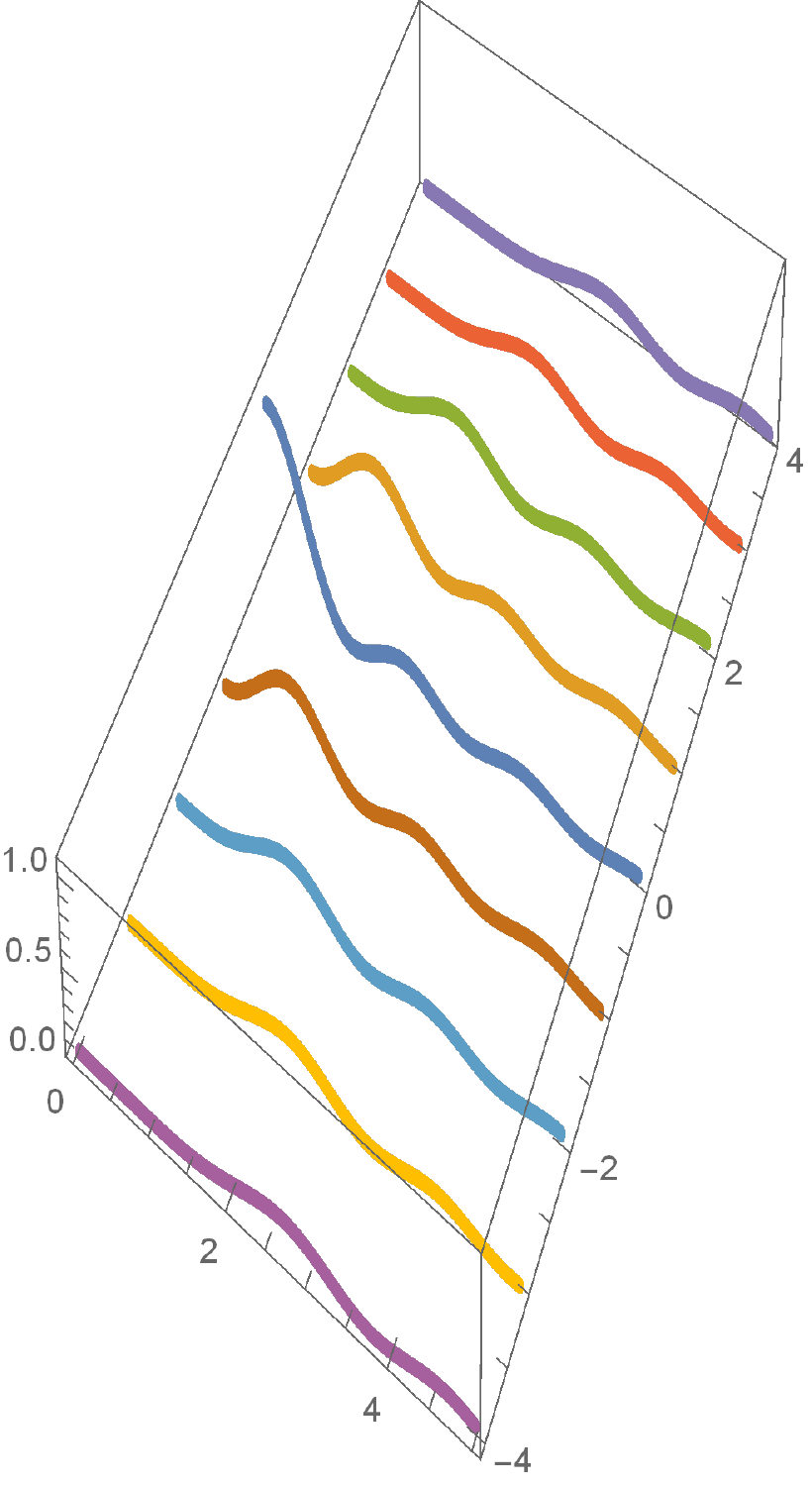}
	\end{subfigure}
	\begin{subfigure}[b]{0.32\textwidth}
		\includegraphics[width=0.7\textwidth]{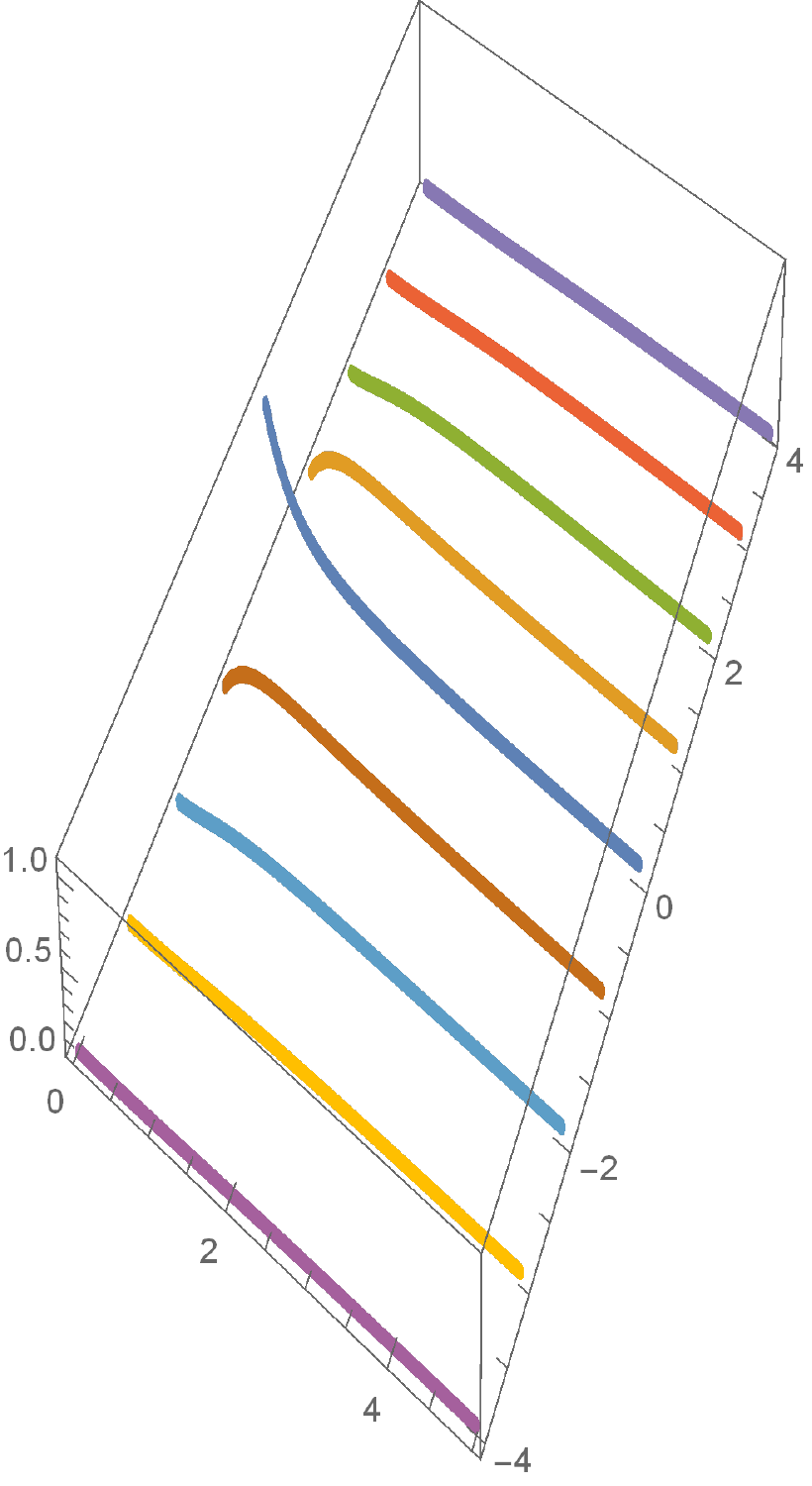}
	\end{subfigure}
	\caption{The time evolution of states on the integers for the exponential and diffusion respectively.}\label{E2}
\end{figure}
We should compare this geodesic calculation with the usual diffusion equation on~$ \Z $ which also gives a time dependent state. Diffusion is defined in terms of a~Lagrangian $ \Delta $ and Lagrangians on graphs have been studied for some time (e.g.,~\cite{T}). We use the special case for diffusion of a density $ f\colon \Z \times [0, \infty) \rightarrow \R$ given by, for $ \lambda >0 $,
\begin{gather*}
\frac{\extd f (n)}{\extd t}= -\lambda (\Delta f ) (n)= -\lambda (2 f (n)-f(n-1)-f(n+1)).
\end{gather*}
This is just (\ref{e55}) with $ f(n)=\alpha_{t}(\delta_{n}) $ but with $ X^{+1}= X^{-1}=-\lambda <0$, i.e., with an ``imaginary'' vector field satisfying $ X^{-1}=\overline{X^{+1} }$. We start at $ t=0 $ with $ f(n)=0 $ for $ n \neq 0 $ and $ f(0)=1 $, and this is just the same as the initial condition for $ \alpha_{t} $ previously. Now the Proposition~\ref{P2} gives the solution for $ f $ as a function of $ t $ in the case $ \lambda=1 $ as
	\begin{gather*}
f=\exp(tX\circ \omega )=\bigg( {}_{0}F_{1}\big(;1;t^{2}\big)z_{0}+ \sum_{n > 0}\dfrac{{}_{0}F_{1} \big(;n+1;t^{2}\big)}{n!}\big( t^{n}z_{n }+t^{n}z_{-n }\big)\bigg){\rm e}^{-2t}\in \C \Z.
\end{gather*}
We plot this for $ -4\leq n\leq 4 $ and $ 0\leq t\leq 5$ as before in the second graph in Figure~\ref{E2}. For both the exponential and the diffusion we have one real parameter, $ \big|X^{+1}\big| $ and $ \lambda $ respectively. We can see from the graphs that the behaviour of the states in the two cases is different, with the exponential giving ``damped oscillations'' and the diffusion giving a monotonic decrease at $ z_{0} $.

\section[The exponential map on quantum SU\_2]{The exponential map on quantum $\boldsymbol{SU_{2}}$}

We use the matrix quantum group $ \C_{q}[SU_{2}] $ as given by Woronowicz \cite{WoroMatl} and quantum enveloping algebra $ U_{q}({\mathfrak{su}}_{2}) $ as given in \cite{PPKa, EKSk}. There is a dual pairing between $ H=\C_{q}[SU_{2}] $ and \mbox{$ H'=U_{q}({\mathfrak{su}}_{2}) $}. (We just say dual pairing as $ H $ is infinite-dimensional and we need to be careful about duals.)

\begin{Definition}
For $ q\in \mathbb{C}^{*} $ with $ q^{2}\neq -1 $, we define the quantum group $ \C_{q}[SU_{2}]$ to have generators $ a$, $b$, $c$, $d $ with relations
\begin{gather*}
\begin{split}
& b a=q a b, \qquad ca = q a c, \qquad d b=q b d,\qquad d c=q c d, \qquad c b=b c,\\
& d a - a d = q\big(1-q^{-2}\big)b c, \qquad a d-q^{-1}b c=1.
\end{split}
\end{gather*}
	This is a Hopf algebra with coproduct, antipode and counit
	\begin{gather*}
		\Delta(a)= a \tens a+b\tens c, \qquad 	\Delta(b)=b \tens d+a \tens b, \qquad 	\Delta(c)=c \tens a+d \tens c, \\
\Delta(d)=d\tens d+c\tens b,\qquad
		S(a)=d, \qquad 	S(b)=-qb, \qquad 	S(c)= -q^{-1}c, \qquad S(d)= a,\\
		\epsilon (a)=\epsilon (d)=1, \qquad \epsilon (b)=\epsilon (c)=0.
	\end{gather*}
		This is a Hopf $ * $-algebra with $a^{*} = d$, $d^{*} = a$, $c^{*}=-qb$ and $b^{*}=-q^{-1} c$ for $ q $ real. We define a~grade on monomials in generations by $ |a|=|c|=1$ and $|b|=|d|=-1$.
	\end{Definition}

\begin{Definition}
	$ U_{q}({\mathfrak{su}}_{2}) $ has generators $ X_{+}$, $X_{-}$, $q^{\pm \frac{H}{2}}$, where we have relations
	\begin{gather*}
	q^{-\frac{H}{2}} q^{\frac{H}{2}} = q^{\frac{H}{2}} q^{-\frac{H}{2}}=1, \qquad q^{\frac{H}{2}} X_{\pm} q^{-\frac{H}{2}}=q^{\pm} X_{\pm}, \qquad [X_{+}, X_{-}]= \dfrac{q^{H}- q^{-H}}{q-q^{-1}},
	\end{gather*}
	and comultiplication, counit and antipode
\begin{gather*}
		\Delta q^{\pm \frac{H}{2}}= q^{\pm \frac{H}{2}} \tens q^{\pm \frac{H}{2}}, \qquad \Delta X_{\pm} = X_{\pm} \tens q^{ \frac{H}{2}}+q^{-\frac{H}{2}} \tens X_{\pm},\\ \epsilon \big(q^{\pm \frac{H}{2}} \big)=1, \qquad \epsilon (X_{\pm})=0,\qquad
	S(X_{\pm})=-q^{\pm} X_{\pm}, \qquad S \big( q^{\pm \frac{H}{2}}\big)= q^{\mp \frac{H}{2}}.
\end{gather*}
As in Definition \ref{D1}, these are dually paired by $ 	\langle \alpha, t^{i} {}_{j}\rangle= \rho (\alpha)^{i} {}_{j} \in \C$ where $ \alpha \in U_{q}({\mathfrak{su}}_{2}) $ and $ t^{1}{}_{1}=a$, $t^{1}{}_{2}= b$, $t^{2}{}_{1}=c$ and $ t^{2}{}_{2}=d $ and $ \rho\colon U_{q}({\mathfrak{su}}_{2}) \rightarrow M_{2}(\C) $ is the representation (where $ r=\sqrt{q}$)
	\begin{gather*}
	\rho(q^{\frac{H}{2}} )=\left( \begin{matrix}
		r& 0\\
		0& \frac{1}{r}
	\end{matrix} \right), \qquad \rho (X_{+})= \left( \begin{matrix}
		0& 1 \\
		0& 0
	\end{matrix} \right), \qquad \rho (X_{-})= \left( \begin{matrix}
		0& 0\\
		1& 0
	\end{matrix} \right).
\end{gather*}
\end{Definition}
\begin{Definition}[\cite{WoroMatl}] The left covariant $ 3 $D calculus for the quantum group $ \C_{q}[SU_{2}] $ has generators $ e^{0} $ and $ e^{\pm} $. The relations are
\begin{gather*}
	e^{\pm} a = q a e^{\pm},\qquad e^{\pm} b = q^{-1} b e^{\pm},\qquad e^{\pm} c = q c e^{\pm},\qquad e^{\pm} d = q^{-1} d e^{\pm},\\
	e^{0} a = q^{2} a e^{0},\qquad e^{0} b = q^{-2} b e^{0},\qquad e^{0} c = q^{2} c e^{0},\qquad e^{0} d = q^{-2} d e^{0},
\end{gather*}
and exterior derivative and the $*$-operator
\begin{gather*}
	\extd a = a e^{0}+q b e^{+},\qquad \extd b = a e^{-}-q^{-2} b e^{0}, \qquad 	\extd c =c e^{0}+ q d e^{+},\qquad \extd d = c e^{-}-q^{-2} d e^{0}, \\	e^{0*}= - e^{0}, \qquad e^{+*}= -q^{-1} e^{-}, \qquad e^{-*}= -q e^{+}.
\end{gather*}
\end{Definition}
Now using $ \omega\colon H\rightarrow \Lambda_{H}^{1} $ from (\ref{e58}) we calculate
\begin{gather} \label{e59}
\omega(a)= q^{-2}e^{0}, \qquad \omega(b) =q^{-1}e^{-}, \qquad \omega(c) = q^{2}e^{+}, \qquad \omega(d) = - e^{0}.
\end{gather}
We define $ e_{0}$, $e_{+}$, $e_{-} $ to be the dual basis of $ e^{0}$, $e^{+}$, $e^{-} $, i.e., $ \langle e_{i}, e^{j}\rangle= \delta_{ij} $. Now every $ e_{i}\circ \omega $ gives a map from $ \C_{q}[SU_{2}]$ to~$ \C $. We shall identify $ e_{i}\circ \omega $ as an element of $ U_{q}({\mathfrak{su}}_{2})$. The first step is to apply $ e_{i}\circ \omega $ to a product.
\begin{Lemma}\label{L1}
		For all $ g,h \in \C_{q}[SU_{2}] $
		\begin{gather*}
		(e_{\pm} \circ \omega)(g h) = (e_{\pm} \circ \omega)(g) \epsilon (h)+q^{-|g|} \epsilon (g) (e_{\pm} \circ\omega)(h),\\
		(e_{0} \circ \omega)(g h) = (e_{0} \circ \omega)(g) \epsilon (h)+q^{-2|g|} \epsilon (g) (e_{0} \circ\omega)(h).
	\end{gather*}
\end{Lemma}
\begin{proof}
	By definition
	\begin{gather*}
	\omega (g h)= \extd \big(g_{(2)} h_{(2)}\big) S^{-1} \big(g_{(1)} h_{(1)}\big)=\omega(g)\epsilon(h)+g_{(2)} \omega (h) S^{-1}\big(g_{(1)}\big).
	\end{gather*}
	Now $ e_{j} g_{(2)} e^{i} S^{-1}\big(g_{(1)}\big)=0$ unless $ i=j $, so we need to show
	\begin{gather*}
	g_{(2)} e^{\pm} S^{-1} \big(g_{(1)}\big)=\epsilon (g) q^{-|g|}e^{\pm}, \qquad g_{(2)} e^{0} S^{-1} \big(g_{(1)}\big)=\epsilon (g) q^{-2|g|}e^{0}.
	\end{gather*}
	It is enough to do this on the generators
	\begin{align*}
	a_{(2)}e^{\pm} S^{-1} (a_{(1)})&= a e^{\pm} S^{-1} (a) + c e^{\pm} S^{-1} (b)\\
	&= a e^{\pm} d-q^{-1} ce^{\pm}b=q^{-1} \big(ad-q^{-1}cb\big)e^{\pm}=q^{-1}e^{\pm},
	\end{align*}
	and similarly for $ e^{0} $ and $ b$, $c$, $d $.
\end{proof}

We can use Lemma \ref{L1} to identify the coproduct of $ e_{i} \circ \omega $, where the linear map $ g\mapsto \epsilon(g) $ is just $ 1\in U_{q}({\mathfrak{su}}_{2}) $. To do this we need to identify the map $ g\mapsto q^{s|g|}\epsilon(g) $.
\begin{Lemma}\label{L2}
	For $ s \in \R$ and $ g\in \C_{q}[SU_{2}] $ we have $\langle q^{sH}, g\rangle= q^{s|g|} \epsilon (g) $.
\end{Lemma}
\begin{proof}As $ \Delta q^{s H}=q^{s H} \tens q^{sH} $ where $ \big\langle q^{sH},hg\big\rangle= \big\langle q^{s H},h\big\rangle \big\langle q^{s H}, g \big\rangle $
we only have to check the formula on the generators, and this is $ \big\langle q^{s H}, t^{i} {}_{j}\big\rangle= \rho \big(q^{s H} \big)^{i} {}_{j} $ using $ \rho\big(q^{s H}\big) = \left( \begin{smallmatrix}
	q^{s} & 0\\
	0& q^{-s}
\end{smallmatrix}
\right)$.
\end{proof}

\begin{Proposition}\label{P5}
	In $ U_{q}({\mathfrak{su}}_{2}) $ we have $ (e_{i}\circ \omega )(h) =	\ev (\nu_{i} \tens h)$ where $ \nu_{i} \in U_{q}({\mathfrak{su}}_{2}) $ is given by, where $ r=\sqrt{q} $,
	\begin{gather*}
		\nu_{0}=e_{0}\circ \omega=\frac{1-q^{-2H}}{q^{2}-1}, \qquad \nu_{+}=e_{+}\circ\omega= r^{3} q^{\frac{-H}{2}} X_{-}, \qquad 	\nu_{-}=e_{-}\circ \omega= r^{-1}q^{\frac{-H}{2}} X_{+}.
	\end{gather*}
\end{Proposition}
\begin{proof}
	First we check that $\nu_{i}=e_{i}\circ \omega $ on the generators
		\begin{gather*}
		\rho (\nu_{0})= \left( \begin{matrix}
		q^{-2}& 0\\
			0& -1
		\end{matrix} \right), \qquad\rho (\nu _{+} )= \left( \begin{matrix}
			0& 0 \\
			q^{2}& 0
		\end{matrix} \right), \qquad \rho (\nu_{-})= \left( \begin{matrix}
			0& q^{-1} \\
			0& 0
		\end{matrix} \right)  .
	\end{gather*}
Thus $ \nu_{0} $ applied to $ b$, $c $ gives zero, and this is also true for $ e_{0}\circ\omega$ by~(\ref{e59}).
Also $ \nu_{+} $ applied to~$a$,~$b$,~$d $ gives zero, and this is also true for $ e_{+}\circ \omega$ by~(\ref{e59}). Lastly $ \nu_{-} $ applied to $a$, $c$, $d $ gives zero, and this is also true for $ e_{-}\circ \omega$. We are left with the cases, for
\begin{gather}
		(e_{0}\circ\omega)(a)=e_{0}\big(q^{-2} e^{0}\big)= q^{-2}= \langle \nu_{0}, a \rangle, \qquad
	(e_{0}\circ\omega)(d)=e_{0}\big({-}e^{0}\big)= -1= \langle \nu_{0}, d \rangle,\nonumber\\
	(	e_{+}\circ\omega)(c)= q^{2}= \langle \nu_{+},c\rangle, \qquad
	(e_{-}\circ\omega)(b)= q^{-1}= \langle \nu_{-},b \rangle .\label{e62}
\end{gather}
Now we need to show that we get equality on products of generators, which we show by induction. Suppose that is the dual $ e_{i} \circ \omega =\nu_{i}$ when applied to products of $ \leq n $ generators.
If $ g$, $h $ are product of $ \leq n $ generators, then
\begin{gather*}
	\nu_{\pm}(g h)=		\nu_{\pm}(g) \epsilon (h) +\big\langle q^{-H},g\big\rangle \nu_{\pm} (h) =		\nu_{\pm}(g) \epsilon (h) +q^{-|g|} 	\nu_{\pm} (h),
\end{gather*}
which is $ (e_{\pm} \circ \omega)(gh) $ by Lemma~\ref{L1}. Also from Lemma~\ref{L2}
\begin{gather*}
	\nu_{0}(g h)= \dfrac{1}{q^{2}-1}\epsilon (g)\epsilon (h)-\dfrac{1}{q^{2}-1}\big\langle q^{-2H},g h\big\rangle = \dfrac{\epsilon (g)\epsilon (h)}{q^{2}-1}\big(1-q^{-2|g|} q^{-2|h|} \big),
\end{gather*}
whereas
\begin{align*}
	(e_{0} \circ \omega)(g h)&= \frac{\epsilon (h)}{q^{2}-1}\big(\epsilon (g)-q^{-2|g|}\epsilon (g)\big)+ \frac{q^{-2|g|}}{q^{2}-1} \big(\epsilon (h)-q^{-2|h|}\epsilon (h)\big)\epsilon (g)\\
	&= \frac{\epsilon (g)\epsilon (h)}{q^{2}-1}\big(1-q^{-2|g|}q^{-2|h|}\big)=\nu_{0}(g h)
\end{align*}
as required.
\end{proof}

The previously calculated exponentials in this case give a formal exponential of a linear combination of the~$ \nu$ in $U_{q}({\mathfrak{su}}_{2}) $. We consider that there is little reason to just write out such a formal sum. However, we can calculate the time evolution of certain states on $ \C_{q}[SU_{2}] $. We use states of the form $ \psi_t(a)=\phi(m(t) a m(t)^{*}) $ where $ (m(t) )$ is defined in Theorem~\ref{prp2tothm}. First we look at the real vector field $ X= {\rm i} e_{0} $.

\begin{Proposition}\label{P4}
For any monomial $ y $ in the generators $ a$, $b$, $c$, $d $ we have $ \nu_{0}(y)=-[-|y|]_{q^{2}} \epsilon (y) $ using the $ q^{2} $ integer $ [n]_{q^{2}}= \big(1-q^{2n}\big)/\big(1-q^{2}\big) $, and
	\[\exp ({\rm i} t \nu_{0})(y)=\exp\big({-}{\rm i} t[-|y|]_{q^{2}}\big) \epsilon(y).\]
\end{Proposition}
\begin{proof}
	The first result is given by using the definition of $ \nu_{0} $ in Proposition~\ref{P5} with Lemma~\ref{L2}. Then $ \nu_{0}^{2}(y)=\nu_{0}(y_{(1)}) \nu_{0} (y_{(2)})= [-|y_{(1)}|]_{q^{2}} [-|y_{(2)}|]_{q^{2}} \epsilon(y_{(1)}) \epsilon(y_{(2)})$. As $ \epsilon $ annihilates $ b $ and $ c $ we must have $ y $ containing no $ b$, $c $ for this to be nonzero. As $ \Delta a= a \tens a +b \tens c$ and $ \Delta d= d\tens d+c \tens b$ we have $ \nu_{0}^{2}(y)=\nu_{0}(y)^{2} $. This continues with higher powers.
\end{proof}
\begin{Proposition}
	For the real vector field $ X={\rm i}e_{0} $ we have $($summing over monomials in $ m(0))$
	\[m(t)=m(0)\exp \big({\rm i}t [-|m(0)|]_{q^{2}}\big).\]
\end{Proposition}
\begin{proof}
	We have from Theorem \ref{prp2tothm} and Proposition \ref{P4}
\[
		m(t)=m(0)_{(1)} \exp (-{\rm i}t e_{0} \circ \omega)\big(m(0)_{(2)}\big)=m(0)_{(1)} \exp \big({\rm i}t [-|m(0)_{(2)}|]_{q^{2}}\big) \epsilon \big(m(0)_{(2)}\big).
\]
	Note that for a monomial $ m(0) $, we have $ |m(0)_{(2)}|=|m(0)| $ giving the answer.
\end{proof}

 We now look at the $ X=\gamma e_{+}+\delta e_{-} $ case. In the previous examples we had $ \psi_0(a) =\phi (m(0) a m(0)^{*}) \allowbreak =\epsilon(a)$. However, because of the complexity of the calculation we will simply calculate~$ m(t) $ for~$ m(0)$ a generator.
 \begin{Proposition} \label{P7}
 For $ X= \gamma e_{+} + \delta e_{-} $ we calculate $ m(t) $ for $ m (0) $ being generator to be
 \begin{gather*}
 		m(0)=a, \qquad m(t)=	a \cosh \big(t\sqrt{q \gamma \delta}\big)-b\frac{q^{2} \gamma}{\sqrt{q \gamma \delta}} \sinh \big(t\sqrt{q \gamma \delta}\big),\\
 	m(0)=b, \qquad m(t)= 	b\cosh \big(t\sqrt{q \gamma \delta}\big)-a\frac{q^{-1} \delta}{\sqrt{q \gamma \delta}} \sinh \big(t\sqrt{q \gamma \delta}\big),\\
 	m(0)=c, \qquad m(t)= c\cosh \big(t\sqrt{q \gamma \delta}\big)-d\frac{q^{2} \gamma}{\sqrt{q \gamma \delta}} \sinh \big(t\sqrt{q \gamma \delta}\big),\\
 	m(0)=d, \qquad m(t)= d\cosh \big(t\sqrt{q \gamma \delta}\big)-c\frac{q^{-1} \delta}{\sqrt{q \gamma \delta}} \sinh \big(t\sqrt{q \gamma \delta}\big).
 \end{gather*}
 \end{Proposition}
\begin{proof}
	We get the corresponding element of the dual $ x=X\circ \omega= \gamma \nu_{+} + \delta \nu_{-} $. Then $ x^{n}(h)=x\big(h_{(1)}\big) x^{n-1}\big(h_{(2)}\big) $, so from~(\ref{e62})
	\begin{gather*}
		x^{n}(a)=q^{-1} \delta x^{n-1}(c), \qquad 	x^{n}(b)=q^{-1} \delta x^{n-1}(d), \\ 	x^{n}(c)=q^{2} \gamma x^{n-1}(a), \qquad x^{n}(d)=q^{2} \gamma x^{n-1}(b).
	\end{gather*}
	These give
	\begin{gather*}
	x^{\text{odd}}(a)=	x^{\text{odd}}(d)=0, \qquad x^{\text{even}}(b)=	x^{\text{even}}(c)=0, \\
	x^{2n}(a)=(q \gamma \delta )^{n}=x^{2n}(d), \qquad
		x^{2n+1}(b)=q ^{n-1}\gamma^{n} \delta ^{n+1}, \qquad 	x^{2n+1}(c)=q ^{n+2}\gamma^{n+1} \delta ^{n},
	\end{gather*}
	so for the formal exponential $ {\rm e}^{-tx} (h)= \sum_{n\geq 0} \dfrac{(-tx)^{n}}{n!}(h)$
	\begin{gather*}
		{\rm e}^{-tx} (a)= \sum_{n\geq 0} \dfrac{t^{2n}(-x)^{2n}}{2n!}(a)= \sum_{n\geq =0} \dfrac{t^{2n}(q \gamma \delta)^{n}}{2n!}=\cosh \big(t\sqrt{q \gamma \delta}\big)={\rm e}^{-tx} (d),\\
		{\rm e}^{-tx} (b)= \sum_{n\geq 0} \dfrac{t^{2n+1}(-x)^{2n+1}}{(2n+1)!}(b)= -\sum_{n\geq 0} \dfrac{t^{2n+1} q^{n-1} \gamma ^{n+1}\delta^{n}}{(2n+1)!}=\frac{-q^{-1} \delta}{\sqrt{q \gamma \delta}}\sinh \big(t\sqrt{q \gamma \delta}\big),\\
		{\rm e}^{-tx} (c)= \sum_{n\geq 0} \dfrac{t^{2n+1}(-x)^{2n+1}}{(2n+1)!}(c)= -\sum_{n\geq 0} \dfrac{t^{2n+1} q^{n+2} \gamma ^{n+1}\delta^{n}}{(2n+1)!}=\frac{-q^{2} \gamma}{\sqrt{q \gamma \delta}}\sinh \big(t\sqrt{q \gamma \delta}\big),
	\end{gather*}
and we use Theorem \ref{prp2tothm} to get the answer.
\end{proof}
\begin{Example}
	For $ X= \gamma e_{+} + \delta e_{-} $ to be real by Proposition~\ref{P3} we have $ \gamma^{*} =-q^{-1} \delta $, so $ \sqrt{q \gamma \delta}=\sqrt{-q^{2}| \gamma|^{2}} =\pm {\rm i} q | \gamma|$. Then a special case of Proposition~\ref{P7} gives
\begin{gather*}
	m(0)=a, \qquad m(t)	=a \cosh ({\rm i} q | \gamma|t)-b\frac{q^{2}\gamma}{i q | \gamma|} \sinh ({\rm i} q | \gamma|t)= a \cos (q | \gamma|t)-b\frac{ q \gamma } {| \gamma |} \sin (q | \gamma|t).
\end{gather*}
Using an exponential map on an initial state will give a time evolution preserving the normalisation as long as we use a real invariant vector field. We now check this in our case of $ \gamma e_{+}+ \delta e_{-} $. We begin with the Haar integral on $ \C_{q}[SU_{2}] $ which is zero on all basis elements $ a^{n} b^{r} c^{s} $ and $ d^{n} b^{r} c^{s} $ except
\begin{align*}
	\phi ((bc)^{r})=\frac{(-1)^{r}q^{r}}{[r+1]_{q^{2}}}.
\end{align*}
Starting from $ m(0)=a $ we find from $ \phi(m(t) m(t)^{*}) $ that $ \psi_{t}(1)=\frac{q^{2}}{1+q^{2}} $ which is independent of $ t $ as required.
\end{Example}

\section[Sweedler-Taft algebra]{Sweedler--Taft algebra}
We will now look at the Sweedler--Taft algebra $ H $ of dimension~$ 4 $. This Hopf algebra does not have a normalised Haar integral, and since we use the integral in finding the ``state'' we shall get an algebraic construction which is nothing like the $ C^{*} $-algebra framework. The Sweedler--Taft algebra~\cite{Taft} is a unital algebra with generators $ x$, $t $ and relations
\begin{gather*}
		t^{2}=1,\qquad x^{2}=0, \qquad xt=-tx,
\end{gather*}
so it is 4-dimensional with basis $ \{1,t,x,xt\} $. The following operations make it into a Hopf algebra
	\begin{gather*}
		\Delta 1 =1 \tens 1, \qquad\Delta t= t \tens t, \qquad \Delta x = x \tens t + 1 \tens x, \qquad \Delta(t x)= t x \tens 1 + t \tens tx,\\ \epsilon(t)=\epsilon(1)=1, \qquad \epsilon(x)=\epsilon(t x)=0, \qquad St=t, \qquad Sx= t x.
	\end{gather*}
We make the Sweedler--Taft algebra into a Hopf $ * $-algebra by $ t^{*} =t$ and $ x^{*}= x $.

There is a unique 2D bicovariant calculus with right ideal $ \mathcal{I}\subset H^{+} = \ker\epsilon $ generated by $ x-xt $ and the $ * $ operation above has $ S((x-xt)^{*})= x-xt $ and so gives a $ * $-calculus \cite{BMbook}. We take the basis $ e^{1}=[x]\in H^{+}/\mathcal{I}=\Lambda ^{1}$ and $ e^{2}=[t-1] $ of $ \Lambda ^{1} $ and relations
\begin{gather*}
\begin{split}
&	\extd t =t e^{2}, \qquad \extd x= x e^{2}+e^{1}, \qquad \extd (x t)=t e^{1},\\
&	e^{2} t	=- t e^{2}, \qquad e^{1} t	= t e^{1}, \qquad e^{1} x= x e^{1}, \qquad e^{2} x= -x e^{2}-2e^{1},\\
&	e^{1} \wedge e^{1}=e^{2} \wedge e^{2}=0,\qquad e^{1} \wedge e^{2}= e^{2} \wedge e^{1}, \qquad \extd e^{2}=0, \qquad \extd e^{1}=- e^{1} \wedge e^{2}.
\end{split}
\end{gather*}
Then $ \extd t^{*}=-e^{2} t$ gives $ e^{2}{}^{*}=-e^{2} $ and $ \extd x^{*}= x e^{2}+ e^{1}$ gives $ e^{1}{}^{*}=- e^{1} $. We introduce a basis of left invariant right vector fields $ e_{1} $ and $ e_{2} $ dual to $ e^{1},e^{2} \in \Lambda^{1}$. From~(\ref{e58}) we find $ \omega (t)= -e^{2} $, $ \omega (x)= -e^{1} $ and $ \omega (tx)= -e^{1} $. Now we find $ e_{i}\circ \omega \in H' $. We get $ e_{i}\circ \omega (1)=0 $ and
\begin{gather*}
e_{2}\circ \omega (t)=-1, \qquad e_{2}\circ \omega (x)=e_{2}\circ \omega (tx)=0, \qquad e_{1}\circ \omega (t)=0, \\
 e_{1}\circ \omega (x)=e_{1}\circ \omega (tx)=-1.
\end{gather*}
The dual basis elements corresponding to $ 1$, $t$, $x$, $tx $ are $ \delta_{1}$, $\delta_{t}$, $\delta_{x}$, $\delta_{tx} $ and their products are
\begin{align*}
	\delta_{1}. \delta_{1}= \delta_{1}, \qquad \delta_{t}. \delta_{t}= \delta_{t}, \qquad \delta_{x}. \delta_{t}= \delta_{x}=\delta_{1}. \delta_{x}, \qquad \delta_{tx}. \delta_{1}= \delta_{tx}=\delta_{t}. \delta_{tx},
\end{align*}
and any other product is zero. Then $ e_{1}\circ \omega=-\delta_{t} $ and $ e_{2}\circ \omega=-\delta_{x} -\delta_{tx} $. The vector field $ X=a e_{1}+b e_{2}$ for $ a,b \in \C $ has
\begin{gather*}
X \circ \omega = -a \delta_{t}-b (\delta_{x} +\delta_{tx} ),\\
	(X \circ \omega)^{2}	 = a ^{2} \delta^{2}_{t}+b^{2}(\delta_{x} +\delta_{tx} )^{2}+ab \delta_{t}(\delta_{x} +\delta_{tx} )+ab (\delta_{x} +\delta_{tx} )\delta_{t}= -a (X \circ \omega),
\end{gather*}
so adding a time parameter $ s $ as usual ($ t $ being used already)
\begin{align*}
	\exp(-s(X \circ \omega))& = \epsilon-s (X \circ \omega)+\frac{(-s (X \circ \omega))^{2}}{2!}+\frac{(-s (X \circ \omega))^{3}}{3 !}+\cdots \\ & =\epsilon-\dfrac{{\rm e}^{as}-1}{a}(X \circ \omega),
\end{align*}
so from Theorem \ref{prp2tothm}
\begin{gather}\label{e71}
	m(s)=m(0)+ m(0)_{(1)} \dfrac{1-{\rm e}^{s a}}{a} (X \circ \omega)(m(0)_{(2)}).
\end{gather}
We now turn to the inner product and definition of real vector fields. The Haar integral~\cite{BMbook} is given by $ \phi (t x)= \lambda $ where $ \lambda $ is arbitrary and $ \phi $ of all other basis elements, including~$ 1 $, to be zero. For Proposition~\ref{prop2} we require $~\phi $ to be Hermitian, so $ \lambda $ is imaginary. Applying this $ \phi $ in the formula for the inner product gives all left invariant vector fields being real, as $ \phi (1)=0 $. Note that this definition of reality corresponds to preserving the inner product in Proposition~\ref{prido}, rather than what in this case is the stronger condition on duals in Proposition~\ref{P3}.
\begin{Example}
Set $ m(0)=t+x $ and $ X=a e_{1}+be_{2} $ as above. Then by (\ref{e71})
\begin{align*}
	m(s)& =t+x+\dfrac{1-{\rm e}^{s a}}{a} (t (X \circ \omega) (t) +1 (X \circ \omega(x))+x (X \circ \omega)(t))\\
& =t+x+\dfrac{1-{\rm e}^{s a}}{a} (-tb -xb-a).
\end{align*}
To calculate the value of the ``state'' $ \psi_s $ in $ \phi(m(t) a m(t)^{*}) $ we set $ m(s)= m_{1} 1+m_{t} t +m_{x} x +m_{tx} tx $ for $ m _{i}\in \C $ then for the Haar measure $ \phi $ the map $ a \mapsto \phi(m a m^{*}) $ is the element of the dual
\begin{gather*}
	 \lambda \delta_{1} (-m _{1} \overline{m_{tx}} +m_{t} \overline{m_{x}}-m_{x} \overline{m_{t}}+m_{t x}\overline{m_{1}} )+ \lambda \delta_{t } (m _{1} \overline{m_{x}} -m_{t} \overline{m_{tx}}-m_{x} \overline{m_{1}}+m_{t x}\overline{m_{t}} ) \\
\qquad{} + \lambda \delta_{x} (-m _{1} \overline{m_{t}} +m_{t} \overline{m_{1}} ) +\lambda \delta_{tx} \big(|m _{1}|^{2} - |m _{t}|^{2} \big).
\end{gather*}
Substituting the values for $ m(s) $ gives
\begin{gather*}
	\psi_s = \lambda (\delta_{t} -\delta_{x} )\left({\rm e}^{sa}-{\rm e}^{s\overline{a}}+\big|{\rm e}^{sa}-1\big|^{2}\left(\frac{\overline{b}}{\overline{a}}-\frac{b}{a}\right)\right)\\
\hphantom{\psi_s =}{}
+\lambda \delta_{tx} \left(\big|{\rm e}^{sa}-1\big|^{2} \left(1-\left|\frac{b}{a}\right|^{2}\right)-1-\dfrac{b}{a} \big({\rm e}^{sa}-1\big)- \frac{\overline{b}}{\overline{a}}\big({\rm e}^{s\overline{a}}-1\big)\right).
\end{gather*}
\end{Example}

\subsection*{Acknowledgements}

We would like to thank the editor and referees for many useful comments. Computer algebra and graphs were done on Mathematica.

\pdfbookmark[1]{References}{ref}
\LastPageEnding

\end{document}